\documentclass[11pt]{amsart}
\usepackage{amsmath, color}
\usepackage{amsfonts}
\usepackage{latexsym}
\usepackage{ifpdf}
\usepackage{graphicx,amssymb,lineno}
\usepackage{amssymb}
\usepackage{mathrsfs}
\usepackage{cite}
\usepackage{extarrows}
\usepackage{ulem}
\allowdisplaybreaks[4]
\numberwithin{equation}{section}

\newtheorem{theorem}{{\bf Theorem}}[section]
\newtheorem{proposition}{{\bf Proposition}}[section]

\newtheorem{lemma}{{\bf Lemma}}[section]
\newtheorem{corollary}{{\bf Corollary}}[section]
\newtheorem{remark}{{\bf Remark}}[section]

\date{}
\begin{document}
\title{Tau functions of the charged free bosons}
\author{Naihuan Jing}
\address{School of Mathematics, South China University of Technology,
Guangzhou, Guangdong 510640, China}
\address{Department of Mathematics, North Carolina State University, Raleigh, NC 27695, USA}
\email{jing@math.ncsu.edu}
\author{Zhijun Li}
\address{School of Mathematics, South China University of Technology,
Guangzhou, Guangdong 510640, China}
\email{zhijun1010@163.com}
\thanks{{\scriptsize
\hskip -0.6 true cm MSC (2010): Primary: 17B37; Secondary: 58A17, 15A75, 15B33, 15A15.
\newline Keywords: tau functions, boson-boson correspondence, vertex operator algebras, symmetric functions.
\newline Corresponding author: zhijun1010@163.com
}}

\maketitle
\begin{abstract}
 We study bosonic tau functions 
 in relation with the charged free bosonic fields. It is proved that up to a constant
 the only tau function in the Fock space $\mathcal{M}$ is the vacuum vector, and
 some tau functions  were given in the completion $\widetilde{\mathcal{M}}$ using Schur functions. 
 We also give a new proof of Borchardt's identity and obtain several $q$-series identities by using
 the boson-boson correspondence.
\end{abstract}
\section{Introduction}
The $bc$ fermionic fields and charged free bosons are important in conformal field theory partly due to their
applications in 
 free field realizations of affine Lie algebras and related algebras (cf. \cite{ FMS1985} and references therein). In this regard the $bc$ fermionic fields
 provide  $\mathfrak{gl}_{\infty}$-modules, level one $\widehat{\mathfrak{gl}}_{\infty}$-modules \cite{KRR2013,KaVa1987}, and
 a free field realization of $\mathcal{W}_{1+\infty}$ algebra in positive integral central charge 
  \cite{FKRW1995}. The $bc$ fermionic fields also provide a Lie theoretic approach to the celebrated KP (Kadomtsev-Petviashvili) hierarchy\cite{DJKM1981a,DJKM1981B}.
 One can formulate the Hirota equation for the KP hierarchy as:
\begin{align}\label{e:KP}
\mathrm{Res}_{z}b(z)\otimes c(z)(\tau\otimes \tau)=0,
\end{align}
where $b(z)$ and $c(z)$ are fermionic fields satisfying \eqref{e:bc}. The solutions $\tau$ of \eqref{e:KP} are called the KP tau functions. Under
the boson-fermion correspondence, the tau functions can be expressed as certain symmetric functions. In particular, one
can get polynomial soliton solutions of the KP hierarchy
\cite{KRR2013, CWW2015} in the Fock space representations.

The charged free bosons $\varphi(z)$ and $\varphi^{*}(z)$ enjoy similar properties to the $bc$ fermionic fields in conformal field theory. They also
 provide $\mathfrak{gl}_{\infty}$-modules, level $-1$ $\widehat{\mathfrak{gl}}_{\infty}$-modules,
 and a free field realization of $\mathcal{W}_{1+\infty}$ algebra with negative integral central charge \cite{KaVa1987,Li2011,Wang1997}. By the boson-boson correspondence or FMS bosonization \cite{BaFl2015,FMS1985,Li2011}, one can also consider the corresponding PDEs and the bosonic tau functions as solutions of the equation \cite{VOS2012}
\begin{align}\label{sp2}
\mathrm{Res}_{z}\varphi^{*}(z)\otimes \varphi(z)(\tau\otimes \tau)=0.
\end{align}

Our first result is to show that
 the only bosonic tau function in the Fock space $\mathcal{M}$ is the vacuum vector $|0\rangle$ up to constant. So the natural
 tau functions in this case lie in the completion $\widetilde{\mathcal M}$. In fact we will show that $\tau= \exp\left(\sum_{i\geq 0,j>0}c_{i,j}\varphi^{*}_{-i}\varphi_{-j}\right)\cdot |0\rangle$ in the space $\widetilde{\mathcal{M}}$ are bosonic KP-like tau functions.

 The difficulty to derive equivalent forms of tau functions in the completion of the polynomial space appearing in charged free bosons case in contrast to the $bc$ fermionic fields case is partly due to the following heuristic reason: in both cases the action of exponentiating an element of $\mathfrak{gl}_{\infty}$ on the vacuum element gives us tau functions, but only in the fermionic case this exponential is finite.
Our next result is to give explicit formulas for some families of the bosonic tau functions in terms of Schur functions in several sets
 of variables. As an example, we obtain that a special bosonic hierarchy equation is a harmonic equation.

The relation between the boson-boson correspondence and the $bc$ fermionic fields has been studied by Wang \cite{Wang1997}.
With the help of the correspondence, we give a new proof of Borchardt's identity
\begin{align*}
\det\left(\frac{1}{(z_{i}-w_{j})^{2}}\right)_{1\leq i,j\leq n}=\det\left(\frac{1}{z_{i}-w_{j}}\right)_{1\leq i,j\leq n}\mathrm{perm}\left(\frac{1}{z_{i}-w_{j}}\right)_{1\leq i,j\leq n}.
\end{align*}
By calculating the $q$-dimension of the (bosonic and fermionic) Fock space of $bc$ fermionic fields and charged free bosons
with combinatorial method and using the results from \cite{FF1992,Wang1997}, we obtain families $q$-series identities($l\geq 0$):
\begin{align*}
&\sum_{m\geq 0}\frac{q^{m}}{(q;q)_{m}(q;q)_{m+l}}=\frac{1}{(q;q)^{2}_{\infty}}\sum_{s\geq 0}(-1)^{s}q^{\frac{s(s+1)}{2}+sl},\\
&\sum_{m\geq0}\frac{q^{m^{2}+(l+1)m}}{(q;q)_{m}(q;q)_{m+l}}=\frac{1}{(q;q)_{\infty}}\sum_{s\geq 0}(-1)^{s}q^{\frac{s(s+1)}{2}+sl},\\
&\sum_{m\geq0}\frac{q^{m^{2}+ml}}{(q;q)_{m}(q;q)_{m+l}}= \frac{1}{(q;q)_{\infty}}.
\end{align*}
The special case of $l=0$ is the famous Euler identity.

The paper is organized as follows. In section 2, we review some basic results of simple lattice vertex algebras, Schur functions and $bc$ fermionic fields. In section 3, we first recall some relations and bosonization of charged free bosons, then we give some concrete forms of tau functions in the completion of the polynomial space of the derived Hirota equations, and prove Borchardt's identity in terms of
the boson-boson correspondence. We have shown that a particular bosonic hierarchy equation is a harmonic equation. In section 4, we consider the $q$-dimension in two ways to derive some $q$-series identities.
\\
\\

\section{Preliminaries}
\subsection{Simple lattice vertex algebras}
The lattice vertex operator algebras are important algebraic structures generalizing finite dimensional simple Lie algebras
 with numerous applications \cite{FLM, KRR2013, DoLe1993, LiLep, Xu2016} (see also \cite{Ada2003}). We briefly review simple lattice vertex algebras here.

 Let $L$ be an integral lattice spanned by the basis $\gamma^i$ with a bilinear form $(\cdot\mid \cdot) : L\times L\rightarrow \mathbb{Z}$
 and $\mathfrak{h}=\mathbb{C}\otimes _{\mathbb{Z}} L$ its complexification.
The \textit{twisted group algebra} $\mathbb{C}_{\epsilon}[L]$ is the algebra generated by $e^{\gamma}$ ($\gamma\in L$) with the twisted multiplication
such that
\begin{align}
\notag e^{\gamma}e^{\delta}=\epsilon(\gamma,\delta)e^{\gamma+\delta},\qquad\qquad (\gamma,\delta\in L),
\end{align}
where $\epsilon:L\times L\rightarrow \{\pm 1\}$ is the 2-cocycle such that:
\begin{gather}
\notag \epsilon(\gamma,\delta)\epsilon(\delta,\gamma)=(-1)^{(\gamma\mid \delta)+(\gamma\mid \gamma)(\delta\mid \delta)}.
\end{gather}

 Let $\widehat{\mathfrak{h}}=\mathfrak{h}[t,t^{-1}]+\mathbb{C}c$ be the affinization of $\mathfrak{h}$, and
 $S$ the symmetric algebra of the commutative subalgebra $\widehat{\mathfrak{h}}^{-}=\sum_{j< 0}\mathfrak{h}\otimes t^{j}$. Write $h_{j}:=h\otimes t^{j}$.
  The \textit{lattice vertex algebra} $V_{L}=S\otimes \mathbb{C}_{\epsilon}[L]$ is the space generated by elements of the form $\gamma^1_{i_{1}}\dots \gamma^1_{i_{j}}\dots \gamma^n_{k_{1}}\dots\gamma^n_{k_{l}}e^{\gamma}$ for $\gamma\in L,\gamma^k\in \mathfrak{h}$ with parity $$p\big{(}\gamma^1_{i_{1}}\dots \gamma^1_{i_{j}}\dots \gamma^n_{k_{1}}\dots\gamma^n_{k_{l}}e^{\gamma}\big{)}\equiv(\gamma\mid\gamma)   \mod 2$$
and the state-field correspondence is given by $(n_{i}\in \mathbb{Z}_{+},h^i\in \mathfrak{h},\gamma\in L)$:
\begin{gather}
\notag Y(h^1_{-n_{1}-1}h^2_{-n_{2}-1}\dots e^{\gamma},z)=:\partial^{(n_{1})}h^1(z)\partial^{(n_{2})}h^2(z)\dots \Gamma_{\gamma}(z):~,
\end{gather}
where
\begin{align}
&h(z)=\sum_{n\in \mathbb{Z}}h_nz^{-n-1},~~
\partial^{(n)}h(z)=\frac{\partial^{n}}{n!}h(z),~~ \\ &\Gamma_{\gamma}(z)=Y(e^{\gamma},z)=e^{\gamma}z^{\gamma_{0}}\exp\left(\sum_{j<0}\frac{z^{-j}}{-j}\gamma_{j}\right)\exp\left(\sum_{j>0}\frac{z^{-j}}{-j}\gamma_{j}\right),
\end{align}
where $:\ \ :$ is the normal ordered product \cite{FLM}.

 For convenience we collect commutation relations of the fields $h(z) (h\in \mathfrak{h})$
 and $Y(e^{\gamma},z)$ as follows:
\begin{align}
\notag &[h(z), h^{\prime}(w)]=(h\mid h^{\prime})\partial_w\delta(z-w) \qquad &(h,h^{\prime}\in \mathfrak{h}),\\
\notag &[h(z), Y(e^{\gamma},w)]=(h\mid \gamma)Y(e^{\gamma},w)\delta(z-w)   \qquad  &(h\in \mathfrak{h}, \gamma\in L),\\
\notag &[Y(e^{\gamma},z), Y(e^{\delta},w)]=\sum_{n\geq 0}\frac{Y(e_{n}^{\gamma}e^{\delta},w)}{n!}\partial_{w}^{n}\delta(z-w)  \qquad &(\gamma,\delta\in L),
\end{align}
where  $[a,b]=ab-(-1)^{p(a)p(b)}ba $,~~
 $\delta(z-w)=\sum_{i\in \mathbb{Z}}z^{-i-1}w^{i}$, and $e^{\gamma}_ne^{\delta}$ is the $n$th product. 
 
\subsection{Schur polynomials}
The \textit{complete symmetric function} $h_{k}(x)$ in the variables $x_1,x_2,\cdots$ is defined by \cite{Mac1995}
\begin{align*}
\sum_{k=0}^{\infty}h_{k}(x)z^{k}=\prod^{\infty}_{i=1}\frac{1}{1-x_iz}.
\end{align*}
To each partition $\lambda=\{\lambda_1\geq \lambda_2\geq\dots\geq \lambda_l> 0\}$ we associate the \textit{Schur function} $s_\lambda(x)$ defined by
\begin{align*}
s_{\lambda}(x)=\det\left(h_{\lambda_{i}-i+j}(x)\right)_{1\leq i,j\leq l}.
\end{align*}
Introduce the variable $t_n$, where $t_n=\frac{1}{n}\sum^{\infty}_{i=1}x^{n}_{i}$ for all $n\geq 1$.
Then $h_k(x)$ is a polynomial $S_{k}(t)$ in the $t_n$ called the {\it elementary Schur polynomial} \cite{KRR2013}, and can be
defined by the series expansion
\begin{align}\label{e:Schur}
\sum_{k=0}^{\infty}S_{k}(t)z^{k}=\exp\left(\sum^{\infty}_{n=1}t_{n}z^{n}\right).
\end{align}
Explicitly, one has
\begin{align}
\notag&S_k(t)=0~~~~\text{for}~~~~k<0,~~S_0(t)=1,\\
\label{e:Schur1}&S_k(t)=\sum_{k_1+2k_2+3k_3+\dots=k}\frac{t^{k_1}_1}{k_1!}\frac{t^{k_2}_2}{k_2!}\frac{t^{k_3}_3}{k_3!}\cdots~~\text{for} ~~k>0.
\end{align}
The Schur polynomial $S_{\lambda}(t)$ associated to partition $\lambda$ is given by the Jacobi-Trudi formula:
\begin{align}
S_{\lambda}(t)=\det\left(S_{\lambda_{i}-i+j}(t)\right)_{1\leq i,j\leq l(\lambda)}.
 \end{align}

\subsection{The $bc$ fermionic fields}
 Recall that the $bc$ fermionic fields are
\begin{align}
b(z)=\sum_{i\in \mathbb{Z}}b(i)z^{-i},~~~~c(z)=\sum_{i\in \mathbb{Z}}c(i)z^{-i-1}
\end{align}
with the commutation relations
\begin{align}\label{e:bc}
\{b(i),c(j)\}=\delta_{i,-j},~~~~\{b(i),b(j)\}=\{c(i),c(j)\}=0,
\end{align}
where $\{A,B\}=AB+BA$.
The operator product expansions (OPE) of $b(z)$ and $c(z)$ are 
\begin{align}\label{sp4}
b(z)c(w)\sim \frac{1}{z-w},~~c(z)b(w)\sim \frac{1}{z-w}. 
\end{align}

Let $\mathcal{F}$ be the Fock space spanned by negative (resp. non-positive) modes of $c(z)$ (resp. $b(z)$),
i.e., $\mathcal F$ is generated by the vacuum vector $|0\rangle$ subject to the relations 
\begin{align}\label{eq3}
b(i+1)|0\rangle=c(i)|0\rangle=0,~~~~i\geq 0.
\end{align}

 Then \cite{Li2011,KaVa1987,Wang1997}
\begin{align}
\notag j^{bc}(z)=:c(z)b(z):=\sum_{n\in \mathbb{Z}}j^{bc}_{n}z^{-n-1}
\end{align}
is a free boson with commutation relations
\begin{align}
[j^{bc}_{m},j^{bc}_{n}]=m\delta_{m,-n},~~~~~~[j^{bc}_{m},b(n)]=-b(m+n),~~~~[j^{bc}_{m},c(n)]=c(m+n).
\end{align}
We then have the charge decomposition of $\mathcal{F}$ according to eigenvalues of $j^{bc}_{0}$:
\begin{align*}
\mathcal{F}=\bigoplus_{l\in \mathbb{Z}}\mathcal{F}^{l}.
\end{align*}
Denoted by $\overline{\mathcal{F}}$ the Fock space generated by the fields $\partial b(z)$ and $c(z)$ with the vacuum vector $|0\rangle$, and $\overline{\mathcal{F}}$ also has the charge decomposition: $\overline{\mathcal{F}}=\bigoplus_{l\in \mathbb{Z}}\overline{\mathcal{F}}^{l}$,
according to $j^{bc}_{0}$-eigenvalues.

 We further consider \cite{Wang1997}
\begin{align*}
T^{bc}(z)=:\partial b(z)c(z):=\sum_{n\in \mathbb{Z}}L_{n}z^{-n-2}
\end{align*}
which is a Virasoro field with the commutation relations
\begin{align}
[L_{m},L_{n}]&=(m-n)L_{m+n}-\frac{m^{3}-m}{6}\delta_{m,-n},\\
\label{ch2-3}[L_{m},b(n)]&=-(m+n)b(m+n),~[L_{m},c(n)]=-nc(m+n).
\end{align}

The boson-fermion correspondence
\cite{Fr1981, DKM1982, KRR2013, Jing1991}
realizes the $bc$ fermionic fields by the lattice vertex operators
\begin{align}\label{ch2-1}
 b(z)=Y(e^{\alpha},z),~~~~c(z)=Y(e^{-\alpha},z),
\end{align}
where $(\alpha|\alpha)=1$. Thus
\begin{align}\label{ch2-2}
T^{bc}(z)=\frac12(:\alpha(z)\alpha(z):+\partial \alpha(z)).
\end{align}
It is known \cite{KaVa1993,KRR2013,Li2011} that
\begin{align}\label{equ1}
\mathrm{Res}_{z}(b(z)\otimes c(z))(\tau\otimes \tau)=0
\end{align}
gives the Hirota bilinear equation associated to the KP 
hierarchy. The identification
between $\mathcal{F}^{l}$ and $e^{-l\alpha}\mathbb{C}[t_1, t_2, \cdots]$ \cite{KRR2013,Jing1991} implies that the Schur polynomials are examples of the tau functions (\ref{equ1}). The tau functions are also related to Hurwitz numbers and 2-Toda hierarchies (cf. \cite{ZZ}). Certain twisted form of the KP hierarchy \cite{DJKM1981B} can also be formulated with help of the fermionic vertex operator
algebra and Schur Q-functions (cf. \cite{Jing1991, DLWY2009}). Further generalizations to Hall-Littlewood functions
are also known \cite{Jing1995, WW}.

\section{Charged free bosons}\label{sec1}
We first recall the charged free bosons\cite{KaVa1987,Li2011,Wang1997} and the boson-boson correspondence (FMS bosonization) \cite{FMS1985,Wang1997}. With the representation of the Heisenberg fields we can describe the bosonic KP hierarchy of PDEs and the embedding from the completion of the Fock space of charged free bosons to the completion $\mathcal{B}$ of a polynomial algebra (\ref{sp1}). We then obtain some tau functions in $\mathcal{B}$ and give a new proof of Borchardt's identity. 
\subsection{The relations}
The charged free bosons 
are defined by
\begin{align}
\varphi(z)=\sum_{i\in \mathbb{Z}}\varphi_{i}z^{-i-1},~~~~\varphi^{*}(z)=\sum_{i\in \mathbb{Z}}\varphi^{*}_{i}z^{-i},
\end{align}
with the commutation relations
\begin{align}\label{eq2}
[\varphi_{i},\varphi^{*}_{j}]=\delta_{i,-j},~~~~[\varphi_{i},\varphi_{j}]=[\varphi^{*}_{i},\varphi^{*}_{j}]=0.
\end{align}
The nontrivial  OPE relations are 
\begin{align}\label{sp3}
\varphi(z)\varphi^{*}(w)\sim \frac{1}{z-w},~~\varphi^{*}(z)\varphi(w)\sim -\frac{1}{z-w}.
\end{align}

Let $\mathcal{M}$ be the Fock space of the charged free bosons generated by the vacuum vector $|0\rangle$ satisfying
\begin{align}\label{eq4}
\varphi_{i}|0\rangle=\varphi^{*}_{i+1}|0\rangle=0,~~i\geq 0.
\end{align}
Then $\mathcal{M}$ has the following basis
\begin{align}\label{eq5}
\{\varphi^{n_{l}}_{-i_{l}}\dots\varphi^{n_{1}}_{-i_{1}}\varphi^{*m_{k}}_{-j_{k}}\dots\varphi^{*m_{1}}_{-j_{1}}|0\rangle|
i_{l}>\dots>i_{1}>0,j_{k}>\dots>j_{1}\geq0,n_{s}\geq 0,m_{s}\geq 0\}.
\end{align}
As a vector space, $\mathcal{M}\simeq\mathbb C[\varphi_{-i-1}, \varphi^*_{-i}|i\in\mathbb Z_+]|0\rangle$. We will also need the completion space
$\widetilde{\mathcal{M}}=\mathbb C[[\varphi_{-i-1}, \varphi^*_{-i}|i\in\mathbb Z_+]]|0\rangle$.

\begin{remark} Charged free bosons are usually called the $\beta$-$\gamma$ system, and can also be studied
in a symplectic fermionic vertex operator superalgebra \cite{CFM2014, ACJ2014}.
\end{remark}

 It is known \cite{FKRW1995,Wang1997} that $\mathcal{M}$ is also a module for the algebra $\mathcal{W}_{1+\infty,-1}$ under the action 
\begin{align}\label{e:c1}
J^{l}(z)=:\varphi(z)\partial^{l}\varphi^{*}(z):,
\end{align}
where $ J^{l}(z)=\sum_{k\in \mathbb{Z}}J^{l}_{k}z^{-k-l-1}$.
Therefore, $J^{0}(z)=\sum_{k\in \mathbb{Z}}J^{0}_{k}z^{-k-1}$ is a free bosonic field with commutation relations
\begin{align}\label{eq7}
[J^{0}_{i},J^{0}_{j}]=-i\delta_{i,-j},~~~~[J^{0}_{i},\varphi_{j}]=-\varphi_{i+j},~~~~[J^{0}_{i},\varphi^{*}_{j}]=\varphi^{*}_{i+j}.
\end{align}
The space $\mathcal{M}$ also decomposes itself as a sum of eigenspaces of $-J^{0}_{0}$ (charge decomposition): 
\begin{align*}
\mathcal{M}=\bigoplus_{l\in \mathbb{Z}}\mathcal{M}^{l},
\end{align*}
where $\mathcal{M}^{l}=\{x\in \mathcal{M}|J_0^0 x=-lx\}$.
Similarly,  $J^{1}(z)=\sum_{k\in \mathbb{Z}}J^{1}_{k}z^{-k-2}$ is a Virasoro field with the commutation relations
\begin{align}\label{eq8} [J^{1}_{i},J^{1}_{j}]=(i-j)J^{1}_{i+j}+\frac{i^{3}-i}{6}\delta_{i,-j},~~[J^{1}_{i},\varphi_{j}]=-j\varphi_{i+j}~~,[J^{1}_{i},\varphi^{*}_{j}]=-(i+j)\varphi^{*}_{i+j}.
\end{align}
Then we have the degree decomposition of $\mathcal{M}$ according to the eigenvalues of the operator $J^{1}_{0}$:
\begin{equation}\label{e:deg}
\mathcal{M}=\bigoplus_{n=0}^{\infty}\mathcal{M}_{n},
\end{equation}
where $M_n=\{x\in \mathcal{M}|J_0^1 x=nx\}$, i.e., it is the span of the vectors in the form \eqref{eq5} such that $n=i_1n_1+\cdots+i_1n_1+j_1m_1+\cdots j_km_k$.

Using the method in \cite{Ang2017}, we have the following result.
\begin{proposition}\label{pro1}
Let $\Omega_{U}=\sum_{i\in \mathbb{Z}}\varphi^{*}_{i}\otimes \varphi_{-i}=\mathrm{Res}_{z}\varphi^{*}(z)\otimes \varphi(z)$. If $\tau\in \mathcal{M}$ satisfies the Hirota equation
\begin{align}\label{eq9}
\Omega_{U}(\tau\otimes \tau)=0,
\end{align}
then $\tau=|0\rangle$ up to a constant.
\end{proposition}
\begin{proof}
It follows from definition (\ref{eq4}) that the vacuum vector $|0\rangle$ is a solution of (\ref{eq9}).
Suppose $\tau\neq |0\rangle$ is a solution of \eqref{eq9} and a sum of monomials in the form (\ref{eq5}). Let $N>0$ be the largest integer such that $\varphi_{-N}$ appears in $\tau$, then $\tau$ can be written in the form
\begin{align*}
\sum^{m}_{k=0}\varphi^{k}_{-N}P_{k}(\varphi_{-N+1},\dots,\varphi_{-1};\dots\varphi^{*}_{-2},\varphi^{*}_{-1},\varphi^{*}_{0})|0\rangle
\end{align*}
where $P_m\neq 0$  $(m\geq 1)$ and $P_k$ ($k\geq 0$) are linear combinations of the basis elements \eqref{eq5} such that
the largest $n$ for which $\varphi_{-n}$ appears in $P_k$ is $\leq N-1$. Then we have
\begin{align*}
\Omega_{U}(\tau\otimes \tau) 
=&\varphi^{*}_{N}\otimes \varphi_{-N}(\tau\otimes \tau)+\sum_{i<N}\varphi^{*}_{i}\otimes \varphi_{-i}(\tau\otimes \tau)\\ =&\sum^{m}_{k=0}-k\varphi^{*k-1}_{-N}P_{k}|0\rangle\otimes \sum^{m}_{k=0}\varphi^{k+1}_{-N}P_{k}|0\rangle+\sum_{i<N}\varphi^{*}_{i}\otimes \varphi_{-i}(\tau\otimes \tau).
\end{align*}
Note that the second summand $\sum_{i<N}\varphi^{*}_{i}\otimes \varphi_{-i}(\tau\otimes \tau)$ contains at most
$\varphi^{m}_{-N}$ in the right of the tensor products, there are no other terms to cancel the nonzero term $\varphi^{m-1}_{-N}P_{m}|0\rangle\otimes m\varphi^{m+1}_{-N}P_{m}|0\rangle$. The contradiction shows that $\Omega_U(\tau\otimes \tau)\neq 0$.
\end{proof}
Solutions of the Hirota equation are called {\it tau functions}. By the above proposition one needs to
focus on tau functions in the completion $\widetilde{\mathcal{M}}=\mathbb{C}[[\varphi_{-i-1},\varphi^{*}_{-i}]]_{i\geq 0}|0\rangle$ (see \eqref{eq5}).
\begin{proposition}\label{pro2} The function
$\tau= \exp\left(\sum_{i\geq 0,j>0}c_{i,j}\varphi^{*}_{-i}\varphi_{-j}\right)\cdot |0\rangle\in \widetilde{\mathcal{M}}$ (finitely many $c_{i,j}\neq 0$) is a solution of (\ref{eq9}).
\end{proposition}
\begin{proof} It follows from (\ref{eq2}) that
\begin{align*}
[\Omega_{U},1\otimes \varphi^{*}_{m}\varphi_{n}&+\varphi^{*}_{m}\varphi_{n}\otimes 1]
=\sum_{i\in \mathbb{Z}}[\varphi^{*}_{i},\varphi^{*}_{m}\varphi_{n}]\otimes \varphi_{-i}+\sum_{i\in \mathbb{Z}}\varphi^{*}_{i}\otimes [\varphi_{-i},\varphi^{*}_{m}\varphi_{n}]\\
=&\sum_{i\in \mathbb{Z}}([\varphi^{*}_{i},\varphi^{*}_{m}]\varphi_{n}+\varphi^{*}_{m}[\varphi^{*}_{i},\varphi_{n}])\otimes \varphi_{-i}
+\sum_{i\in \mathbb{Z}}\varphi^{*}_{i}\otimes ([\varphi_{-i},\varphi^{*}_{m}]\varphi_{n}+\varphi^{*}_{m}[\varphi_{-i},\varphi_{n}])\\
=&-\varphi^{*}_{m}\otimes \varphi_{n}+\varphi^{*}_{m}\otimes \varphi_{n}=0.
\end{align*}
Then we have
\begin{align*}
\Omega_{U}(\tau\otimes \tau)=&\exp\left(\sum_{i\geq 0,j>0}c_{i,j}\varphi^{*}_{-i}\varphi_{-j}\right)\otimes \exp\left(\sum_{i\geq 0,j>0}c_{i,j}\varphi^{*}_{-i}\varphi_{-j}\right)(\Omega_{U}(|0\rangle\otimes |0\rangle))\\
=&0.
\end{align*}
\end{proof}

\begin{remark} The assignment $E_{ij}\longrightarrow :\varphi_i\varphi^*_j:$ also provides a level $-1$ representation of the infinite-dimensional
Lie algebra $\widehat{\mathfrak{gl}}(\infty)$. Therefore the bosonic tau functions are in the orbit of $\widetilde{\mathrm{GL}}(\infty)|0\rangle$.
\end{remark}

\subsection{Bosonization }
The {\it Friedan-Martinec-Shenker (FMS) bosonization} \cite{FF1992,Wang1997, FMS1985}
provides a boson-boson realization of the charged free bosons $\varphi(z),~\varphi^{*}(z)$ in terms of a lattice vertex algebra. Let
$L=\mathbb Z\alpha+\mathbb Z\beta$ be the 2-dimensional lattice, where
$
(\alpha\mid\beta)=0,~~~~ (\alpha\mid\alpha)=-(\beta\mid\beta)=1.
$
On the lattice vertex algebra associated to $L$ consider the vertex operators 
\begin{align}\label{e:B4}
\varphi(z)=Y(e^{-\alpha-\beta},z), \qquad 
\varphi^{*}(z)=Y(\alpha_{-1}e^{\alpha+\beta},z), 
\end{align}
and the associated 2-cocycle is
$\varepsilon(\alpha,\alpha)=\varepsilon(\alpha,\beta)=-\varepsilon(\beta,\alpha)=-\varepsilon(\beta,\beta)=1.
$
Using \eqref{ch2-1}, the above can be rewritten as
\begin{align}\label{ch3-1}
\varphi(z)=c(z)Y(e^{-\beta},z),~~~~\varphi^{*}(z)=\partial b(z)Y(e^{\beta},z).
\end{align}
It is easy to see that in this case (cf. \eqref{e:c1}) 
\begin{align}
\label{e:c2}J^{0}(z)&=-\beta(z),\\
\label{eq10}J^{1}(z)&=\frac{:\alpha(z)\alpha(z):+\partial \alpha(z)}{2}-\frac{:\beta(z)\beta(z):+\partial \beta(z)}{2}.
\end{align}
\par Introduce the (completed) bosonic Fock space
\begin{align}\label{sp1}
\mathcal{B}=\mathbb{C}[[x,y;p,p^{-1}]],
\end{align}
where $x=(x_{1},x_{2},x_{3},\dots),~y=(y_{1},y_{2},y_{3},\dots),~p=e^{\alpha+\beta}$ and $x_n=\frac1n\alpha_{-n}, y_n=\frac1n\beta_{-n}$.
Since
\begin{align*}
 [\alpha_m,\alpha_n]=m\delta_{m,-n},~~~~ [\beta_m,\beta_n]=-m\delta_{m,-n},~~~~[\alpha_0,e^{\alpha+\beta}]=e^{\alpha+\beta},~~~~ [\beta_0,e^{\alpha+\beta}]=-e^{\alpha+\beta},
\end{align*}
the Heisenberg fields $\alpha(z)=\sum_{n\in \mathbb{Z}}\alpha_nz^{-n-1}$, $\beta(z)=\sum_{n\in \mathbb{Z}}\beta_nz^{-n-1}$ act on $\mathcal{B}$ as follows $(n>0)$ :
\begin{align*}
&\alpha_n=\partial x_n,~~~~\alpha_{-n}=nx_n,~~~~\alpha_0=p\partial p,\\
&\beta_n=-\partial y_n,~~~~\beta_{-n}=ny_n,~~~~\beta_0=-p\partial p.
\end{align*}

\par Then we have an embedding of $\widetilde{\mathcal{M}}$ into $\mathcal{B}$ by $|0\rangle\rightarrow 1$ and
\begin{align}\label{eq11}
\notag &\varphi^{*}(z)=p\exp\left(\sum_{n>0}\left(x_{n}+y_{n}\right)z^{n}\right)\left(\sum_{k>0}kx_{k}z^{k-1}+\sum_{k>0}\partial x_{k}z^{-k-1}+p\partial p z^{-1}\right)\\
&\qquad\qquad\qquad\qquad\qquad \exp\left(-\sum_{n>0}(\partial x_{n}-\partial y_{n})\frac {z^{-n}}{n}\right),\\
&\label{eq32}\varphi(z)=p^{-1}\exp\left(-\sum_{n>0}(x_{n}+y_{n})z^{n}\right)\exp\left(\sum_{n>0}(\partial x_{n}-\partial y_{n})\frac {z^{-n}}{n}\right).
\end{align}
Therefore one can write $\Omega_U=\mathrm{Res}_{z}\varphi^{*}(z)\otimes \varphi(z)$ as an operator on the space $\mathcal B\otimes \mathcal B$. For simplicity we denote $X\otimes Y=X'Y''$,
then
\begin{align}\notag
\Omega_{U}=&\mathrm{Res}_{z}p^{\prime}p^{\prime\prime-1}\exp\left(\sum_{n>0}(x^{\prime}_{n}-x^{\prime\prime}_{n}+y^{\prime}_{n}-y^{\prime\prime}_{n})z^{n}\right)
\left(\sum_{k>0}kx^{\prime}_{k}z^{k-1}+\sum_{k>0}\partial x^{\prime}_{k}z^{-k-1}+p^{\prime}\partial p^{\prime} z^{-1}\right)\\ \label{eq12}
&\exp\left(-\sum_{n>0}(\partial x^{\prime}_{n}-\partial x^{\prime\prime}_{n}-\partial y^{\prime}_{n}+\partial y^{\prime\prime}_{n})\frac{z^{-n}}{n}\right).
\end{align}
We introduce the following new operators over $\mathcal{B}\otimes \mathcal{B}$:
\begin{equation*}
 A=\frac{1}{2}(A^{\prime}-A^{\prime\prime}), \qquad \bar{A}=\frac{1}{2}(A^{\prime}+A^{\prime\prime}).
\end{equation*}
Then
\begin{equation}
 A^{\prime}=A+\bar{A}, \qquad A^{\prime\prime}=\bar{A}-A.
\end{equation}
This gives us
\begin{align*}
\Omega_{U}=&\mathrm{Res}_{z}p^{\prime}p^{\prime\prime-1}\exp\left(\sum_{n>0}(2x_{n}+2y_{n})z^{n}\right)\\
&\left(\sum_{k>0}k(x_{k}+\bar{x}_{k})z^{k-1}+\sum_{k>0}\frac{\partial x_{k}+\partial \bar{x}_{k}}{2}z^{-k-1}+p^{\prime}\partial p^{\prime} z^{-1}\right)
\exp\left(\sum_{n>0}(-\partial x_{n}+\partial y_{n})\frac {z^{-n}}{n}\right).
\end{align*}
\begin{proposition} Let $\tau$ be a tau function in
$\mathbb{C}[[x,y]]$. Then 
\begin{align}\label{eq14}
 \notag &\sum_{i,j\geq0}S_{i}\left(2x+2y\right)\left[\left(j-i\right)\left(x_{j-i}+\bar{x}_{j-i}\right)+\frac{\partial \lambda_{i-j}+\partial \bar{x}_{i-j}}{2}\right]
S_{j}\left(-\widetilde{\partial} \lambda+\widetilde{\partial} \mu\right)\\
&\qquad \exp\left(\sum_{l\geq1}(x_{l}\partial \lambda_{l}+y_{l}\partial \mu_{l})\right)
\tau\left(\bar{x}-\lambda,\bar{y}-\mu\right)\tau\left(\bar{x}+\lambda,\bar{y}+\mu\right)\mid_{\lambda=\mu=0}=0,
\end{align}
where $\widetilde{\partial }\lambda=\left(\partial \lambda_{1},\frac{\partial \lambda_{2}}{2},\frac{\partial \lambda_{3}}{3},\dots\right).$
\end{proposition}

Let's write down some Hirota equations given by \eqref{eq14}. Denote
$\tilde{S}_{n}=S_{n}\left(-\widetilde{\partial }\lambda+\widetilde{\partial} \mu\right)$ and
consider the case where
\begin{align}
\notag x_{i}=\bar{x}_{i}=y_{i}=\bar{y}_{i}=0,~~~~\text{for}~~ i\geq 2.
\end{align}
Then the coefficient of $x_{1}$ in (\ref{eq14}) gives the equation
\begin{align}
\notag &\left(2\overline{x}_{1}\tilde{S}_{2}+\partial \lambda_{1}+\partial\overline{x}_{1}+\overline{x}_{1}\tilde{S}_{1}\partial \lambda_{1}+\tilde{S}_{1}\right)
\tau(\bar{x}-\lambda,\bar{y}-\mu)\tau(\bar{x}+\lambda,\bar{y}+\mu)\mid_{\lambda=\mu=0}=0.
\end{align}

Writing $\bar{x}_{1}=u,~\bar{y}_{1}=v$ and $g(u,v)=\log \tau$, the above gives the differential equation:
\begin{align}
\notag u(-g_{uv}+g_{vv})+g_{u}=0.
\end{align}
Similarly, set
$x_{i}=\bar{x}_{i}=y_{i}=\bar{y}_{i}=0,~~~~\text{for}~~  i\geq 3$
and $\bar{x}_{1}=u,~\bar{y}_{1}=v,~f(u,v)=\log \tau$. Then one has the following differential equation from (\ref{eq14}):
\begin{align}\label{e:hierarchy2}
f_{uu}-2f_{uv}+f_{vv}=0.
\end{align}
Let $t=u-v, s=u+v$, then $\partial_t=\frac12(\partial_u-{\partial_v})$ and
$\partial_s=\frac12({\partial_u}+{\partial_v})$. So the special bosonic
hierarchy equation \eqref{e:hierarchy2} is a harmonic equation:
\begin{equation}\label{e:har}
\tilde{f}_{tt}=0.
\end{equation}
for the function $f=f(u, v)=\tilde{f}(t, s)$. This surprising result will be studied further elsewhere.

\begin{remark}
If $\tau\in\mathbb{C}[[x]]$, we get the so-called $\beta$-reduction \cite{Li2011}
\begin{align*}
 \sum_{i,j\geq0}&S_{i}\left(2x\right)\left[\left(j-i\right)\left(x_{j-i}+\bar{x}_{j-i}\right)+\frac{\partial \lambda_{i-j}+\partial \bar{x}_{i-j}}{2}\right]
S_{j}\left(-\widetilde{\partial }\lambda\right)\\
&\exp\left(\sum_{l\geq1}x_{l}\partial \lambda_{l}\right)
\tau(\bar{x}-\lambda)\tau(\bar{x}+\lambda)\mid_{\lambda=0}=0.
\end{align*}
In particular, the tau function of the $\beta$-reduction 
is also a tau function of (\ref{eq14}),
The converse is not true in general.
\end{remark}

\subsection{Tau functions }
\par We now discuss the tau functions of (\ref{eq14}). First, we give new expressions for the tau functions $\exp\left(\sum^{k}_{j=1}a_{j}\varphi_{-j}\varphi^{*}_{0}\right)\cdot 1$, $\exp\left(a\varphi_{-j}\varphi^{*}_{-1}\right)\cdot 1,~j\geq 1$, then we give formulas for  $\exp\left(a\varphi_{-s}\varphi^{*}_{-t}\right)\cdot 1,~s\geq 1,~t\geq 2$  and $\exp\left(d\varphi_{-j}\varphi^{*}_{-k}\right)\exp\left(c\varphi_{-i}\varphi^{*}_{-l}\right)\cdot 1,~i,j\geq 1,~k\geq l\geq 0$. 
\par  To describe our results, we need the elementary Schur polynomials $S_k(x),~S_k(-x-y),~S_k(x+y)$, which are defined similarly as \eqref{e:Schur} by the following generating functions: 
\begin{align}
&\label{schur1}\sum_{k=0}^{\infty}S_{k}(-x)w^{k}=\exp\left(-\sum^{\infty}_{n=1}x_{i}w^{i}\right),\\
&\sum_{k=0}^{\infty}S_{k}(-x-y)w^{k}=\exp\left(-\sum^{\infty}_{n=1}\left(x_i+y_i\right)w^{i}\right),\\
&\sum_{k=0}^{\infty}S_{k}(x+y)w^{k}=\exp\left(\sum^{\infty}_{n=1}\left(x_i+y_i\right)w^{i}\right).
\end{align}
\begin{lemma}\label{le3}
One has for $n\geq 0$
\begin{align}\label{n2}
\varphi^{*n}_{0}\cdot 1=(-1)^{n}n!p^{n}S_{n}\left(-x\right).
\end{align}
\end{lemma}
\begin{proof} In terms of the variables $x_i$, $y_i$ and $p$, the field operator $\varphi^*(z)$ can be rewritten as (see \eqref{eq11})
\begin{align}\label{eq30}
\varphi^{*}(z)&=p\exp\left(\sum_{i>0}\left(x_{i}+y_{i}\right)z^{i}\right)\left(\sum_{k>0}kx_{k}z^{k-1}+\sum_{k>0}\partial x_{k}z^{-k-1}+p\partial p z^{-1}\right)\\ \label{eqABC}
\notag&\qquad\qquad\qquad\qquad\qquad \exp\left(-\sum_{i>0}(\partial x_{i}-\partial y_{i})\frac {z^{-i}}{i}\right)\\
&=p\exp\left(\sum_{i>0}\left(x_{i}+y_{i}\right)z^{i}\right)(A+B+C),
\end{align}
where we have used $A, B, C$ to denote the respective summands in the second and third factors, e.g.
$A=\sum_{k> 0}kx_{k}z^{k-1}\exp\left(-\sum_{i>0}(\partial x_{i}-\partial y_{i})\frac {z^{-i}}{i}\right)$, etc.

To show \eqref{n2} by induction on $n$, we note that it clearly holds for $n=0,1$. Now suppose \eqref{n2} is valid for $n$, and we want to show it for $n+1$. To this end, set $R=(-1)^{n}n!p^{n}S_{n}\left(-x\right)$ and it follows from induction hypothesis that
\begin{align}\label{eq31}
\varphi^{*}(z)\varphi^{*n}_{0}\cdot 1=p\exp\left(\sum_{i>0}\left(x_{i}+y_{i}\right)z^{i}\right)\left(A+B+C\right)R.
\end{align}
Applying $\exp\left(-\sum_{i\geq 1}(\partial x_{i}-\partial y_{i})\frac{z^{-i}}{i}\right)$ and $\partial_w$ to the series $\exp\left(-\sum_{j\geq 1} x_{j}w^{j}\right)$ respectively, we have that
\begin{align*}
&\exp\left(-\sum_{i\geq 1}(\partial x_{i}-\partial y_{i})\frac{z^{-i}}{i}\right)\exp\left(-\sum_{j\geq 1} x_{j}w^{j}\right)=\frac{1}{1-\frac{w}{z}}\exp\left(-\sum_{j\geq 1} x_{j}w^{j}\right),\\
&\partial_{w}\exp\left(-\sum_{j\geq 1} x_{j}w^{j}\right)=-\sum_{j\geq 1}jx_{j}w^{j-1}\exp\left(-\sum_{j\geq 1} x_{j}w^{j}\right).
\end{align*}
Collecting the coefficients, we get that
\begin{align*}
&\exp\left(-\sum_{i\geq 1}(\partial x_{i}-\partial y_{i})\frac{z^{-i}}{i}\right)S_{n}(-x)=\sum^{n}_{i=0}S_{i}(-x)z^{-(n-i)},\\
&\sum_{i\geq 1}ix_{i}S_{n-i}(-x)=-nS_{n}(-x).
\end{align*}
Consequently the coefficient of $z^{0}$ in $(A+B+C)R$ is 
\begin{align*}  
e_0:=(-1)^{n} n!p^{n}\sum_{k> 0}kx_{k}S_{n+1-k}(-x)=-(n+1)(-1)^{n} n!p^{n}S_{n+1}(-x)=(-1)^{n+1} (n+1)!p^{n}S_{n+1}(-x).
\end{align*}
By (\ref{eq7}) and (\ref{eq10}) it follows that
\begin{align*}
\beta_{i}(\varphi^{*j}_{0}\cdot 1)=-j\varphi^{*j-1}_{0}\varphi^{*}_{i}\cdot 1=0, ~~i\geq 1,
\end{align*}
i.e., $S_k(x+y)$ does not appear in $\varphi^{*n+1}_{0}\cdot 1$ for $k\geq 1$. Note that $\varphi^{*n+1}_{0}\cdot 1$ is
the coefficient of $z^0$ in $\varphi^{*}(z)\varphi^{*n}_{0}\cdot 1$. Using \eqref{eq31} we have that
 \begin{align*}
 \varphi^{*n+1}_{0}\cdot 1=pS_{0}(x+y)e_{0}=(-1)^{n+1} (n+1)!p^{n+1}S_{n+1}(-x),
 \end{align*}
which completes the proof.
\end{proof}
The following result is clear.
\begin{lemma}\label{le2} One has that
\begin{align*}
\sum_{k\geq m}C^{m}_{k}t^{k-m}=\frac{1}{(1-t)^{m+1}}.
\end{align*}
\end{lemma}
\begin{theorem}\label{th1} We have that
\begin{align}
\exp\left(\sum^{s}_{j=1}a_{j}\varphi_{-j}\varphi^{*}_{0}\right)\cdot 1=\frac{1}{1-\sum^{s}_{j=1}a_{j}S_j(-x-y)}\exp\left(-\sum_{n\geq 1}x_{n}(\frac{-\sum^{s}_{j=1}a_{j}S_{j-1}(-x-y)}{1-\sum^{s}_{j=1}a_{j}S_j(-x-y)})^{n}\right).
\end{align}
\end{theorem}
\begin{proof}
First by \eqref{schur1} and \eqref{n2}, we see that
\begin{align*}
\exp\left(a\varphi^{*}_{0}\right)\cdot 1=\exp\left(-\sum_{n\geq 1}x_{n}(-ap)^{n}\right).
\end{align*}
Recalling \eqref{eq32}, we have that
\begin{align*}
\varphi(z)\exp\left(a\varphi^{*}_{0}\right)\cdot 1=&p^{-1}\exp\left(-\sum_{n>0}(x_{n}+y_{n})z^{n}\right)\exp\left(-\sum_{n\geq 1}\left(x_n+\frac{z^{-n}}{n}\right)(-ap)^{n}\right)\\
=&p^{-1}\exp\left(-\sum_{n>0}(x_{n}+y_{n})z^{n}\right)(1-apz^{-1})\exp\left(-\sum_{n\geq 1}x_{n}(-ap)^{n}\right).
\end{align*}
Thus
\begin{align*}
\varphi_{-j}\exp\left(a\varphi^{*}_{0}\right)\cdot 1=\left(p^{-1}S_{j-1}(-x-y)-aS_j(-x-y)\right)\exp\left(-\sum_{n\geq 1}x_{n}(-ap)^{n}\right).
\end{align*}
Since $\varphi(z)$ commutes with $\exp\left(-\sum^{\infty}_{n=1}\left(x_{i}+y_{i}\right)w^{i}\right)$,
i.e., its components $\varphi_j$ commute with $S_{n}(-x-y)$ for all integral $j$. 
Therefore
\begin{align*}
\varphi^m_{-j}\exp\left(a\varphi^{*}_{0}\right)\cdot 1=\left(p^{-1}S_{j-1}(-x-y)-aS_j(-x-y)\right)^m\exp\left(-\sum_{n\geq 1}x_{n}(-ap)^{n}\right).
\end{align*}
Then we have
\begin{align}\notag
&\prod^{s}_{j=1}\exp\left(a_{j}\varphi_{-j}\right)\exp\left(a\varphi^{*}_{0}\right)\cdot 1\\ \label{eq29}
=&\exp\left(p^{-1}\sum^{s}_{j=1}a_{j}S_{j-1}(-x-y)\right)\exp\left(a\sum^{s}_{j=1}a_{j}S_{j}(-x-y)\right)\exp\left(-\sum_{n\geq 1}x_{n}(-ap)^{n}\right).
\end{align}

Comparing the coefficient of $a^kp^0$ on both sides of (\ref{eq29}), we get
\begin{align*}
&\frac{1}{k!}\sum_{\substack{m_1,m_2,\dots,m_{s-1}\geq 0\\m_1+m_2+\dots+m_{s-1}\leq k}}\frac{(a_{s}\varphi_{-s}\varphi^{*}_{0})^{k-\sum^{s-1}_{j=1}m_{j}}}{(k-\sum^{s-1}_{j=1}m_{j})!}
\prod^{s-1}_{j=1}\frac{\left(a_{j}\varphi_{-j}\varphi^{*}_{0}\right)^{m_{j}}}{m_{j}!}\cdot 1\\
=&\sum_{0\leq m\leq k}\frac{\left(\sum^{s}_{j=1}a_{j}S_{j}(-x-y)\right)^{k-m}}{(k-m)!}\frac{\left(-\sum^{s}_{j=1}a_{j}S_{j-1}(-x-y)\right)^{m}}{m!}S_{m}(-x).
\end{align*}
Thus we have
\begin{align*}
&\sum_{k\geq 0}\sum_{\substack{m_1,m_2,\dots,m_{s-1}\geq 0\\m_1+m_2+\dots+m_{s-1}\leq k}}\frac{(a_{s}\varphi_{-s}\varphi^{*}_{0})^{k-\sum^{s-1}_{j=1}m_{j}}}{(k-\sum^{s-1}_{j=1}m_{j})!}
\prod^{s-1}_{j=1}\frac{(a_{j}\varphi_{-j}\varphi^{*}_{0})^{m_{j}}}{m_{j}!}\cdot 1\\
=&\sum_{k\geq 0}\sum_{0\leq m\leq k}C^{m}_{k}\left(\sum^{s}_{j=1}a_{j}S_{j}(-x-y)\right)^{k-m}\left(-\sum^{s}_{j=1}a_{j}S_{j-1}(-x-y)\right)^{m}S_{m}(-x)\\
=&\sum_{m\geq 0}\left(-\sum^{s}_{j=1}a_{j}S_{j-1}(-x-y)\right)^{m}S_{m}(-x)\sum_{k\geq m}C^{m}_{k}\left(\sum^{s}_{j=1}a_{j}S_{j}(-x-y)\right)^{k-m}\\
=&\sum_{m\geq 0}\left(-\sum^{s}_{j=1}a_{j}S_{j-1}(-x-y)\right)^{m}S_{m}(-x)\frac{1}{\left(1-\sum^{s}_{j=1}a_{j}S_{j}(-x-y)\right)^{m+1}}\\
=&\frac{1}{1-\sum^{s-1}_{j=1}a_{j}S_{j}(-x-y)}\exp\left(-\sum_{n\geq 1}x_{n}\left(\frac{-\sum^{s}_{j=1}a_{j}S_{j-1}(-x-y)}{1-\sum^{s}_{j=1}a_{j}S_{j}(-x-y)}\right)^{n}\right),
\end{align*}
where we have used Lemma \ref{le2}.
That is,
\begin{align*}
\exp\left(\sum^{s}_{j=1}a_{j}\varphi_{-j}\varphi^{*}_{0}\right)\cdot 1=\frac{1}{1-\sum^{s}_{j=1}a_{j}S_j(-x-y)}\exp\left(-\sum_{n\geq 1}x_{n}\left(\frac{-\sum^{s}_{j=1}a_{j}S_{j-1}(-x-y)}{1-\sum^{s}_{j=1}a_{j}S_j(-x-y)}\right)^{n}\right).
\end{align*}
\end{proof}
\begin{theorem}\label{th2} For $j\geq 1$, we have
\begin{align} \label{eq26}
&\exp(a\varphi_{-j}\varphi^{*}_{-1})\cdot 1=\frac{1}{1-aS_{j+1}(-x-y)-aS_{1}(x+y)S_{j}(-x-y)}\\
\notag \cdot&\exp\left(-\sum_{n\geq 1}\sum_{0\leq i\leq n}\frac{C^{i}_{n}S^{n-i}_{1}(x+y)(n+i)x_{n+i}}{n}\left(\frac{-aS_{j-1}(-x-y)}{1-aS_{j+1}(-x-y)-aS_{1}(x+y)S_{j}(-x-y)}\right)^n\right).
\end{align}
\end{theorem}
The proof is left in Appendix A.
\begin{corollary} Suppose that
\begin{align}\label{n1}
\left(\sum^{t}_{m=0}S_{m}(x+y)z^{-t-1+m}\right)^{n}=\sum^{nt}_{i=0}f_{i}z^{-n(t+1)+i},
\end{align}
and let
\begin{align*}
p_{n}=\sum^{nt}_{i=0}f_{i}\left(n(t+1)-i\right)x_{n(t+1)-i}
\end{align*}
and
\begin{align*}
A_{s,t}&=\sum^{t}_{m=0}S_{m}(x+y)S_{s+t-m}(-x-y).
\end{align*}
Then we have
\begin{align*}
\exp\left(a\varphi_{-s}\varphi^{*}_{-t}\right)\cdot 1=\frac{1}{1-aA_{s,t}}\exp\left(-\sum_{n\geq 1}\frac{p_{n}}{n}\left(\frac{-aS_{s-1}(-x-y)}{1-aA_{s,t}}\right)^{n}\right),~~s\geq 1,~t\geq 0.
\end{align*}
\end{corollary}
\begin{corollary} Suppose that
\begin{align*}
\left(\left(\sum^{l}_{m=0}S_{m}(x+y)z^{-l-1+m}\right)w_{1}+\left(\sum^{k}_{m=0}S_{m}(x+y)z^{-k-1+m}\right)w_{2}\right)^{n}=\sum^{n(l+k+1)}_{i=0}f_{i}(w_{1},w_{2})z^{-n(l+k+2)+i}
\end{align*}
and let
\begin{align*}
p_{n}=\sum^{n(l+k+1)}_{i=0}f_{i}\left(w_{1},w_{2})(n(l+k+2)-i\right)x_{n(l+k+2)-i}.
\end{align*}
For $f(w_{1},w_{2})=\exp\left(-\sum_{n\geq 1}\frac{p_{n}}{n}\right)$, we have
\begin{align*}
&\exp\left(\sum_{i\geq 1}\frac{\partial x_{i}-\partial y_{i}}{i}z^{-i}\right)f(w_{1},w_{2})\\
=&\left(1-\left(\sum^{l}_{m=0}S_{m}(x+y)z^{-l-1+m}\right)w_{1}-\left(\sum^{k}_{m=0}S_{m}(x+y)z^{-k-1+m}\right)w_{2}\right)f(w_{1},w_{2}).
\end{align*}
Set $A_{s,t}=\sum^{t}_{m=0}S_{m}(x+y)S_{s+t-m}(-x-y)$. For $k\geq l\geq0$,
we have
\begin{align*}
\exp\left(b\varphi^{*}_{-k}\right)\exp\left(a\varphi^{*}_{-l}\right)\cdot 1=f(-ap,-bp)
\end{align*}
and
\begin{align*}
\exp\left(d\varphi_{-j}\varphi^{*}_{-k}\right)\exp\left(c\varphi_{-i}\varphi^{*}_{-l}\right)\cdot 1
=&\frac{1}{1-cA_{i,l}-dA_{j,k}+cA_{i,l}dA_{j,k}-cA_{i,k}dA_{j,l}}\\
&f\Big{(}\frac{-cS_{i-1}(-x-y)+cS_{i-1}(-x-y)dA_{j,k}-dS_{j-1}(-x-y)cA_{i,k}}{1-cA_{i,l}-dA_{j,k}+cA_{i,l}dA_{j,k}-cA_{i,k}dA_{j,l}},\\
&\frac{-dS_{j-1}(-x-y)+dS_{j-1}(-x-y)cA_{i,l}-cS_{i-1}(-x-y)dA_{j,l}}{1-cA_{i,l}-dA_{j,k}+cA_{i,l}dA_{j,k}-cA_{i,k}dA_{j,l}}\Big{)}.
\end{align*}
\end{corollary}

\subsection{Borchardt's identity}
Following \cite{DJKM1981B}, we define the Fock space $\mathcal{M}^{*}$  of the charged free bosons by
\begin{align}\label{e:B3}
\langle 0|\varphi_{i}=\langle 0|\varphi^{*}_{i+1}=0,~~i<0
\end{align}
with the inner product $\langle 0|1|0\rangle=1$, then (cf. \eqref{e:B4})
\begin{align}\label{e:B1}
\langle 0|b(i-1)=\langle 0|c(i)=\langle 0|\beta_i=0,~~i\leq 0.
\end{align}
Recall \eqref{eq4}
we then have
\begin{align}\label{e:B2}
b(i+1)|0\rangle=c(i)|0\rangle=\beta_i|0\rangle=0,~~i\geq 0.
\end{align}
\eqref{e:B1} and \eqref{e:B2} will be proved in Appendix B.

Computing  $\langle 0|\varphi(z_{1})\dots \varphi(z_{n})\varphi^{*}(w_{1})\dots \varphi^{*}(w_{n})|0\rangle$ in two methods (cf. \cite{VOS2012}),
we get the following result.
\begin{proposition}[Borchardt's identity \cite{Bor1855,IKO2004}] For any positive integer $n$
\begin{align*}
\det\left(\frac{1}{(z_{i}-w_{j})^{2}}\right)_{1\leq i,j\leq n}=\det\left(\frac{1}{z_{i}-w_{j}}\right)_{1\leq i,j\leq n}\mathrm{perm}\left(\frac{1}{z_{i}-w_{j}}\right)_{1\leq i,j\leq n}.
\end{align*}
where $\mathrm{perm} A$ is the permanent of square matrix $A$ defined by
\begin{align*}
\mathrm{perm}A=\sum_{\sigma\in S_{n}}a_{1\sigma(1)}a_{2\sigma(2)}\cdots a_{n\sigma(n)}.
\end{align*}
\end{proposition}
\begin{proof} It follows from (\ref{sp4}) and (\ref{sp3}) that
\begin{align}
\label{eq24}\notag\langle0|c(z_{1})\cdots c(z_{n})\partial b(w_{1})\cdots \partial b(w_{n})|0\rangle
&=(-1)^{\frac{n(n-1)}{2}}\sum_{\sigma}(-1)^{sgn(\sigma)}\prod^{n}_{i=1}\frac{1}{(z_{i}-w_{\sigma(i)})^{2}}\\
&=(-1)^{\frac{n(n-1)}{2}}\det\left(\frac{1}{(z_{i}-w_{j})^{2}}\right)_{1\leq i,j\leq n},\\
\label{eq23}\notag\langle 0|\varphi(z_{1})\cdots \varphi(z_{n})\varphi^{*}(w_{1})\cdots \varphi^{*}(w_{n})|0\rangle
&=\sum_{\sigma\in S_{n}}\prod_{1\leq i\leq n}\frac{1}{z_{i}-w_{\sigma(i)}}\\
&=\mathrm{perm}\left(\frac{1}{z_{i}-w_{j}}\right)_{1\leq i,j\leq n}.
\end{align}
By the boson-boson correspondence \eqref{ch3-1} \cite{FF1992,Wang1997}
\begin{align*}
\varphi(z_{1})\cdots \varphi(z_{n})\varphi^{*}(w_{1})\cdots \varphi^{*}(w_{n})|0\rangle=c(z_{1})\cdots c(z_{n})\partial b(w_{1})\cdots \partial b(w_{n})|0\rangle\\
\frac{\prod_{1\leq s,t\leq n}(z_{s}-w_{t})}{\prod_{1\leq i<j\leq n}(z_{i}-z_{j})(w_{i}-w_{j})}
\exp\left(\sum_{m\geq 1}\frac{\beta_{-n}}{n}(w^{m}_{1}+\cdots w^{m}_{n}-z^{m}_{1}\cdots -z^{m}_{n})\right),
\end{align*}
we have
\begin{align}\label{eq25}
\notag&\langle0|\varphi(z_{1})\cdots \varphi(z_{n})\varphi^{*}(w_{1})\cdots \varphi^{*}(w_{n})|0\rangle\\
=&\frac{\prod_{1\leq s,t\leq n}(z_{s}-w_{t})}{\prod_{1\leq i<j\leq n}(z_{i}-z_{j})(w_{i}-w_{j})}\langle0|c(z_{1})\cdots c(z_{n})\partial b(w_{1})\cdots \partial b(w_{n})|0\rangle.
\end{align}
Then the relations (\ref{eq24}), (\ref{eq23}) and (\ref{eq25}) imply that
\begin{align*}
\mathrm{perm}\left(\frac{1}{z_{i}-w_{j}}\right)_{1\leq i,j\leq n}=(-1)^{\frac{n(n-1)}{2}}\frac{\prod_{1\leq s,t\leq n}(z_{s}-w_{t})}{\prod_{1\leq i<j\leq n}(z_{i}-z_{j})(w_{i}-w_{j})}\det\left(\frac{1}{(z_{i}-w_{j})^{2}}\right)_{1\leq i,j\leq n}.
\end{align*}
Using Cauchy's determinant identity
\begin{align*}
\det\left(\frac{1}{z_{i}-w_{j}}\right)_{1\leq i,j\leq n}=(-1)^{\frac{n(n-1)}{2}}\frac{\prod_{1\leq i<j\leq n}(z_{i}-z_{j})(w_{i}-w_{j})}{\prod_{1\leq s,t\leq n}(z_{s}-w_{t})},
\end{align*}
we have
\begin{align*}
\det\left(\frac{1}{(z_{i}-w_{j})^{2}}\right)_{1\leq i,j\leq n}=\det\left(\frac{1}{z_{i}-w_{j}}\right)_{1\leq i,j\leq n}\mathrm{perm}\left(\frac{1}{z_{i}-w_{j}}\right)_{1\leq i,j\leq n}.
\end{align*}
\end{proof}

\section{$q$-dimension}
In this section we first compute $q$-series of some sets by $q$-Pochhammer symbols and then use them to
derive $q$-dimensions of $\mathcal{M}^{l},~\overline{\mathcal{F}}^{l},~\mathcal{F}^{l}$ by combinatorial method.
We also derive some $q$-identities by comparing with the results in \cite{FF1992,Wang1997}.
We adopt the usual convention to denote the $q$-Pochhammer symbols or $q$-integers by
\begin{align*}
&(a;q)_{n}=(1-a)(1-aq)\dots (1-aq^{n-1})=\prod^{n-1}_{j=0}(1-aq^{j}),\\
&(a;q)_{\infty}=\prod^{\infty}_{j=0}(1-aq^{j}),~~~~|q|< 1.
\end{align*}
A partition $\lambda=(\lambda_{1},\lambda_{2},\dots ,\lambda_{l})$ is a sequence of ordered non-negative integers $\lambda_{1}\geq\lambda_{2}\geq\dots \geq\lambda_{l}\geq 0~(or ~0\leq\lambda_{1}\leq\lambda_{2}\leq\dots \leq\lambda_{l})$. The $weight$ of
$\lambda$ is defined by $|\lambda|=\lambda_{1}+\lambda_{2}+\dots+\lambda_{l}$.
 Let $\mathcal{P}$ be a set of partitions with some properties. We define 
\begin{align*}
G(\mathcal{P})=\sum_{\lambda\in \mathcal{P}}q^{|\lambda|}.
\end{align*}
\begin{lemma}\label{pro9}
Let $\mathcal{P}_{k}^{s}=\{\lambda=(\lambda_{1},\lambda_{2},\dots ,\lambda_{s})|k\leq\lambda_{1}\leq\lambda_{2}\leq\dots \leq\lambda_{s}\}$. Then
\begin{align}
G(\mathcal{P}_{k}^{s})=\frac{q^{ks}}{(q;q)_{s}}.
\end{align}
\end{lemma}
\begin{proof} We use induction on $s$. The case of $s=1$ is clear as
$G(\mathcal{P}_{k}^{1})=q^{k}+q^{k+1}+\dots =\frac{q^{k}}{1-q}$. Assume $s-1$ is established, then
\begin{align*}
G(\mathcal{P}_{k}^{s})=q^{k}G(\mathcal{P}_{k}^{s-1})+q^{k+1}G(\mathcal{P}_{k+1}^{s-1})+q^{k+2}G(\mathcal{P}_{k+2}^{s-1})+\dots=\frac{q^{ks}}{(q;q)_{s}}.
\end{align*}
\end{proof}
Similarly we have the following.
\begin{lemma}\label{cor1}
For $\mathcal{D}_{k}^{s}=\{\lambda=(\lambda_{1},\lambda_{2},\dots ,\lambda_{s})|k\leq\lambda_{1}<\lambda_{2}<\dots <\lambda_{s}\}$,
\begin{align}
G(\mathcal{D}_{k}^{s})=\frac{q^{\frac{(2k+s-1)s}{2}}}{(q;q)_{s}}.
\end{align}
\end{lemma}
It is known \cite{Wang1997} that $\mathcal{M}\cong \sum_{l\in\mathbb{Z}}\overline{\mathcal{F}}^{l}\otimes e^{-l\beta}\mathbb{C}[y].$
In terms of the degree operator $J^{1}_{0}$ on $\mathcal{M}$ \eqref{e:deg}, we take
\begin{align}\label{eq20}
\notag &\mathrm{deg}~\varphi^{*}_{-n}=\mathrm{deg}~\varphi_{-n}=\mathrm{deg}~ b(-n)=\mathrm{deg}~ c(-n)=\mathrm{deg}~\alpha_{-n}=\mathrm{deg}~\beta_{-n}=n,\\
&\mathrm{deg}~e^{l\alpha}=\frac{l^{2}-l}{2},~\mathrm{deg}~e^{l\beta}=\frac{-l^{2}-l}{2},~\mathrm{deg}~|0\rangle=0,
\end{align}
in view of \eqref{ch2-3}, \eqref{ch2-2}, \eqref{eq8} and \eqref{eq10}.
We pass the gradation $\mathrm{deg}$ to any subspace $W\subseteq \mathcal{M}$, and denote
$W_{m}=\{w|w\in W,\mathrm{deg}(w)=m\}$. Then we define the $q$-dimension of $W$ as
\begin{align}\label{new1}
 \mathrm{dim}_q(W)=\mathrm{tr}_{W}q^{J^{1}_{0}}=\sum_{m\geq 0}\text{dim}W_{m}q^{m}.
\end{align}
The following result gives the $q$-dimensions for the spaces  $\mathcal{M}^{l}$, $\overline{\mathcal{F}}^{l}$ and $\mathcal{F}^{l}$.
\begin{proposition}\label{char} One has that
\begin{align}
&\mathrm{dim}_q(\mathcal{M}^{l})=\sum_{m\geq 0}\frac{q^{m+\delta_{l\geq 0}l}}{(q;q)_{m}(q;q)_{m+|l|}}  ,\\
&\mathrm{dim}_q(\overline{\mathcal{F}}^{l})=\sum_{m\geq0}\frac{q^{m^{2}+(|l|+1)m+\frac{|l|(|l|+1)}{2}}}{(q;q)_{m}(q;q)_{m+|l|}},\\
\label{sp6}&\mathrm{dim}_q(\mathcal{F}^{l})=\sum_{m\geq0}\frac{q^{m^{2}+m|l|+\frac{l(l+1)}{2}}}{(q;q)_{m}(q;q)_{m+|l|}}.
\end{align}
\end{proposition}
\begin{proof} Let
\begin{align*}
\Phi^{m}={\rm span}\{\varphi_{-i_{1}}\dots\varphi_{-i_{m}}|i_{1}\geq\dots\geq i_{m}\geq 1\};
~~~~\Phi^{*m}={\rm span}\{\varphi^{*}_{-j_{1}}\dots\varphi^{*}_{-j_{m}}|j_{1}\geq\dots\geq j_{m}\geq0\}.
\end{align*}
It follows from (\ref{eq20}) that $\mathrm{dim}_q(\Phi^{m})=G(\mathcal{P}_{1}^{m})=\frac{q^{m}}{(q;q)_{m}}$ and $\mathrm{dim}_q(\Phi^{*m})=G(\mathcal{P}_{0}^{m})=\frac{1}{(q;q)_{m}}$, where
the special cases are $G(\Phi^{0})=G(\Phi^{*0})=q^{0}=1$.
For $l\geq 0$ and using (\ref{eq20}), we then have
\begin{align*}
\mathrm{dim}_q(\mathcal{M}^{l})=\sum_{m\geq 0}\mathrm{dim}_q(\Phi^{*m})\mathrm{dim}_q(\Phi^{m+l})=\sum_{m\geq 0}\frac{q^{m+l}}{(q;q)_{m}(q;q)_{m+l}},
\end{align*}
and for $l< 0$,
\begin{align*}
\mathrm{dim}_q(\mathcal{M}^{l})=\sum_{m\geq 0}\mathrm{dim}_q(\Phi^{*m-l})\mathrm{dim}_q(\Phi^{m})=\sum_{m\geq 0}\frac{q^{m}}{(q;q)_{m}(q;q)_{m-l}}.
\end{align*}
Similarly, set
$B_{k}^{m}={\rm span}\{b(-i_{1})\dots b(-i_{m})|i_{1}>\dots >i_{m}\geq k\}$ and $C^{m}={\rm span}\{c(-j_{1})\dots c(-j_{m})|j_{1}>\dots >j_{m}\geq 1\}$.
So for $l\geq 0$ and counting the degree (\ref{eq20}), we have that
\begin{align*}
&\mathrm{dim}_q(\overline{\mathcal{F}}^{l})=\sum_{m\geq 0}\mathrm{dim}_q({B_{1}^{m}})\mathrm{dim}_q(C^{m+l})=\sum_{m\geq 0}G(\mathcal{D}_{1}^{m})G(\mathcal{D}_{1}^{m+l})
=\sum_{m\geq0}\frac{q^{m^{2}+(l+1)m+\frac{l(l+1)}{2}}}{(q;q)_{m}(q;q)_{m+l}},\\
&\mathrm{dim}_q(\mathcal{F}^{l})=\sum_{m\geq 1}\mathrm{dim}_q({B_{0}^{m}})\mathrm{dim}_q(C^{m+l})=\sum_{m\geq 0}G(\mathcal{D}_{0}^{m})G(\mathcal{D}_{1}^{m+l})
=\sum_{m\geq0}\frac{q^{m^{2}+ml+\frac{l(l+1)}{2}}}{(q;q)_{m}(q;q)_{m+l}}.
\end{align*}
Similar for $l< 0$,
\begin{align*}
&\mathrm{dim}_q(\overline{\mathcal{F}}^{l})=\sum_{m\geq 0}\mathrm{dim}_q({B_{1}^{m-l}})\mathrm{dim}_q(C^{m})=\sum_{m\geq 0}G(\mathcal{D}_{1}^{m-l})G(\mathcal{D}_{1}^{m})
=\sum_{m\geq0}\frac{q^{m^{2}+(-l+1)m+\frac{(-l)(-l+1)}{2}}}{(q;q)_{m}(q;q)_{m-l}},\\
&\mathrm{dim}_q(\mathcal{F}^{l})=\sum_{m\geq 1}\mathrm{dim}_q({B_{0}^{m-l}})\mathrm{dim}_q(C^{m})=\sum_{m\geq 0}G(\mathcal{D}_{0}^{m-l})G(\mathcal{D}_{1}^{m})
=\sum_{m\geq0}\frac{q^{m^{2}-ml+\frac{l(l+1)}{2}}}{(q;q)_{m}(q;q)_{m-l}}.
\end{align*}
\end{proof}
\begin{proposition}\label{char1} \cite{FF1992,Wang1997} One also has that
\begin{align}
& \mathrm{dim}_q(\mathcal{M}^{l})=\frac{1}{(q;q)^{2}_{\infty}}\sum_{s\geq 0}(-1)^{s}q^{\frac{s(s+1)}{2}+(s+\delta_{l\geq 0})|l|},   \\
&\mathrm{dim}_q(\overline{\mathcal{F}}^{l})= \frac{1}{(q;q)_{\infty}}\sum_{s\geq |l|}(-1)^{s+l}q^{\frac{s(s+1)}{2}}.
\end{align}
\end{proposition}
The following result is a consequence of Propositions \ref{char} and \ref{char1}.
\begin{theorem} For $l\geq 0$ one has that
\begin{align}
&\sum_{m\geq 0}\frac{q^{m}}{(q;q)_{m}(q;q)_{m+l}}=\frac{1}{(q;q)^{2}_{\infty}}\sum_{s\geq 0}(-1)^{s}q^{\frac{s(s+1)}{2}+sl},\\
&\sum_{m\geq0}\frac{q^{m^{2}+(l+1)m}}{(q;q)_{m}(q;q)_{m+l}}=\frac{1}{(q;q)_{\infty}}\sum_{s\geq 0}(-1)^{s}q^{\frac{s(s+1)}{2}+sl}.
\end{align}
\end{theorem}
Therefore we have the $q$-series identity 
\begin{align*}
\sum_{m\geq 0}\frac{q^{m}}{(q;q)_{m}(q;q)_{m+l}}=\frac{1}{(q;q)_{\infty}}\sum_{m\geq0}\frac{q^{m^{2}+(l+1)m}}{(q;q)_{m}(q;q)_{m+l}}.
\end{align*}
Recalling that $\mathcal{F}^{l}\cong e^{-l\alpha}\mathbb{C}[x]$ \cite{KRR2013} and $\mathrm{deg}(e^{-l\alpha}x_k)=k+\frac{l(l+1)}{2}$, so we also have
\begin{align}\label{sp5}
\mathrm{dim}_q(\mathcal{F}^{l})=\frac{q^{\frac{l(l+1)}{2}}}{(q;q)_{\infty}}.
\end{align}

The following result follows from (\ref{sp6}) and (\ref{sp5}).
\begin{theorem} One has that
\begin{align}
\sum_{m\geq0}\frac{q^{m^{2}+m|l|}}{(q;q)_{m}(q;q)_{m+|l|}}= \frac{1}{(q;q)_{\infty}}.
\end{align}
In particular,
the famous Euler identity is the case of $l=0$.
\end{theorem}
\section{Appendix A}\label{app1}
In the Appendix we shall prove Eq. (\ref{eq26}) in Theorem \ref{th2}.
\begin{proof} It is enough to show that
\begin{align}
\varphi^{n}_{-j}\varphi^{*n}_{-1}\cdot 1=n!\sum^{n}_{i=0}C^{i}_{n}(-1)^{i}S^{i}_{j-1}(-x-y)\left(S_{j+1}(-x-y)+S_{1}(x+y)S_{j}(-x-y)\right)^{n-i}S^{*}_{i},
\end{align}
where $S^{*}_{n}$ is defined by
\begin{align}\label{eq28}
\exp\left(-\sum_{n\geq 1}\sum_{0\leq i\leq n}\frac{C^{i}_{n}S^{n-i}_{1}(x+y)(n+i)x_{n+i}}{n}z^{n}\right)=\sum_{n\geq 0}S^{*}_{n}z^{n}.
\end{align}
\par As in the proof of Lemma \ref{le3}, we claim that
\begin{align}\label{ap1}
\varphi^{*n}_{-1}\cdot 1=(-1)^{n}n!p^{n}S^{*}_{n}.
\end{align}
This is verified by using the same method.
Set $R=(-1)^{n}n!p^{n}S^{*}_{n}$ and
just as in \eqref{eq30} we have
\begin{align*}
\varphi^*(z)\varphi^{*n}_{-1}\cdot 1=p\exp\left(\sum_{i\geq 1}(\partial x_{i}+\partial y_{i})z^i\right)(A+B+C)R,
\end{align*}
where $A, B, C$ are defined in \eqref{eqABC}.

Applying $\partial_{z}$ to two sides of (\ref{eq28}), we have
\begin{align*}
\sum_{k\geq 1}\sum_{\frac{k}{2}\leq n\leq k}C^{k-n}_{n}S^{2n-k}_{1}(x+y)kx_{k}S^{*}_{s-n}=-sS^{*}_{s}.
\end{align*}
From
\begin{align*}
&\exp\left(-\sum_{i\geq 1}(\partial x_{i}-\partial y_{i})\frac{z^{-i}}{i}\right)\exp\left(-\sum_{n\geq 1}\sum_{0\leq i\leq n}\frac{C^{i}_{n}S^{n-i}_{1}(x+y)(n+i)x_{n+i}}{n}w^{n}\right)\\
=&\frac{1}{1-(z^{-2}+S_{1}(x+y)z^{-1})w}\exp\left(-\sum_{n\geq 1}\sum_{0\leq i\leq n}\frac{C^{i}_{n}S^{n-i}_{1}(x+y)(n+i)x_{n+i}}{n}w^{n}\right),
\end{align*}
we get
\begin{align*}
\exp\left(-\sum_{i\geq 1}(\partial x_{i}-\partial y_{i})\frac{z^{-i}}{i}\right)S^{*}_{n}=\sum^{n}_{i=0}(z^{-2}+S_{1}(x+y)z^{-1})^{i}S^{*}_{n-i}.
\end{align*}

Therefore, the coefficient of $z^{0}$ in $(A+B+C)R$ is
\begin{align}
a_{0}:=(-1)^{n} n!p^{n}\sum_{k> 0}\sum^{n}_{i=0}C^{k-1-i}_{i}S^{2i+1-k}_{1}(x+y)kx_{k}S^{*}_{n-i},
\end{align}
and the $z$-coefficient is 
\begin{align}
a_{1}:=(-1)^{n} n!p^{n}\sum_{k> 0}\sum^{n}_{i=0}C^{k-2-i}_{i}S^{2i+2-k}_{1}(x+y)kx_{k}S^{*}_{n-i}.
\end{align}
Since
\begin{align}
\beta_{i}(\varphi^{*j}_{-1}\cdot 1)=-j\varphi^{*j-1}_{-1}\varphi^{*}_{i-1}\cdot 1=0, ~~i\geq 2,
\end{align}
$S_k(x+y)$ does not appear in $\varphi^{*n+1}_{-1}\cdot 1$ for $k\geq 2$.
Thus
\begin{align*}
\varphi^{*n+1}_{-1}\cdot 1&=p(-1)^{n}n!p^{n}\left(a_{0}S_{1}(x+y)+a_{1}S_0(x+y)\right)\\
&=(-1)^{n}n!p^{n+1}\sum_{k> 0}\sum^{n}_{i=0}(C^{k-1-i}_{i}+C^{k-2-i}_{i})S^{2i+2-k}_{1}(x+y)kx_{k}S^{*}_{n-i}\\
&=(-1)^{n}n!p^{n+1}\sum_{k> 0}\sum^{n}_{i=0}C^{k-1-i}_{i+1}S^{2i+2-k}_{1}(x+y)kx_{k}S^{*}_{n-i}\\
&=(-1)^{n}n!p^{n+1}\sum_{k> 0}\sum_{\frac{k}{2}\leq t \leq k}C^{k-t}_{t}S^{2t-k}_{1}(x+y)kx_{k}S^{*}_{n+1-t}\\
&=(-1)^{n+1}(n+1)!p^{n+1}S^{*}_{n+1}.
\end{align*}
Thus \eqref{ap1} is proved.

By the formula
\begin{align*}
&\exp\left(\sum(\partial x_{i}-\partial y_{i})\frac{z^{-i}}{i}\right)\exp\left(-\sum_{n\geq 1}\sum_{0\leq i\leq n}\frac{C^{i}_{n}S^{n-i}_{1}(x+y)(n+i)x_{n+i}}{n}w^{n}\right)\\
=&\left(1-\left(z^{-2}+S_{1}(x+y)z^{-1}\right)w\right)\exp\left(-\sum_{n\geq 1}\sum_{0\leq i\leq n}\frac{C^{i}_{n}S^{n-i}_{1}(x+y)(n+i)x_{n+i}}{n}w^{n}\right),
\end{align*}
we then have
\begin{align*}
\exp\left(\sum(\partial x_{i}-\partial y_{i})\frac{z^{-i}}{i}\right)S^{*}_{n}=S^{*}_{n}-(z^{-2}+S_{1}(x+y)z^{-1})S^{*}_{n-1}.
\end{align*}
Thus
\begin{align}\label{eq21}
\varphi_{-j}S^{*}_{n}=p^{-1}\left(S_{j-1}(-x-y)S^{*}_{n}-\left(S_{j+1}(-x-y)+S_{1}(x+y)S_{j}(-x-y)\right)S^{*}_{n-1}\right).
\end{align}
Repeatedly using (\ref{eq21}) we then have
\begin{align*}
\varphi^{n}_{-j}S^{*}_{n}=p^{-n}\sum^{n}_{i=0}C^{i}_{n}(-1)^{n-i}S^{i}_{j-1}(-x-y)(S_{j+1}(-x-y)+S_{1}(x+y)S_{j}(-x-y))^{n-i}S^{*}_{i}.
\end{align*}
Thus
\begin{align*}
\varphi^{n}_{-j}\varphi^{*n}_{-1}\cdot 1=(-1)^{n}n!\sum^{n}_{i=0}C^{i}_{n}(-1)^{n-i}S^{i}_{j-1}(-x-y)\left(S_{j+1}(-x-y)+S_{1}(x+y)S_{j}(-x-y)\right)^{n-i}S^{*}_{i}.
\end{align*}
\end{proof}

\section*{Appendix B}
Here we prove \eqref{e:B1}. The proof of \eqref{e:B2} is similar.

In fact, it follows from \eqref{e:c1} and \eqref{e:c2} that
\begin{align*}
\beta_k=-\sum_{j\in\mathbb{Z}}:\varphi_j\varphi^*_{-j+k}:=-\sum_{j\geq 0}\varphi^*_{-j+k}\varphi_j-\sum_{j< 0}\varphi_j\varphi^*_{-j+k},
\end{align*}
it is seen that $\langle 0|\beta_k=0$, for $k\leq 0$ by using \eqref{e:B3}.

Recall \eqref{ch3-1}, we have 
\begin{align*}
\langle 0|\varphi(z)=\langle 0|c(z)Y(e^{-\beta},z).
\end{align*}
Then
\begin{align*}
\langle 0|\sum_{i\in \mathbb{Z}}\varphi_iz^{-i-1}
=&\langle 0|\sum_{n\in\mathbb{Z}}c(n)z^{-n-1}e^{-\beta}z^{-\beta_{0}}
\exp\left(-\sum_{j<0}\frac{z^{-j}}{-j}\beta_{j}\right)\exp\left(-\sum_{j>0}\frac{z^{-j}}{-j}\beta_{j}\right)\\
=&z\langle0|z^{-\beta_{0}}\exp\left(-\sum_{j<0}\frac{z^{-j}}{-j}\beta_{j}\right)\exp\left(-\sum_{j>0}\frac{z^{-j}}{-j}\beta_{j}\right)\sum_{j\in\mathbb{Z}}c(n)z^{-n-1}e^{-\beta}\\
=&\langle0|\exp\left(-\sum_{j>0}\frac{z^{-j}}{-j}\beta_{j}\right)\sum_{j\in\mathbb{Z}}c(n)z^{-n}e^{-\beta}\\
=&\langle0|\sum_{j\in\mathbb{Z}}c(n)z^{-n}e^{-\beta}\exp\left(-\sum_{j>0}\frac{z^{-j}}{-j}\beta_{j}\right),
\end{align*}
where we have used $e^{-\beta}z^{-\beta_{0}}=zz^{-\beta_{0}}e^{-\beta}$ and the properties of lattice vertex algebras.
Thus
\begin{align}\label{e:ap1}
\langle 0|\sum_{i\geq 0}\varphi_iz^{-i-1}=\langle0|\sum_{j\in\mathbb{Z}}c(n)z^{-n}e^{-\beta}\exp\left(-\sum_{j>0}\frac{z^{-j}}{-j}\beta_{j}\right).
\end{align}
Comparing coefficients of $z^{n}$ for $n\geq 0$  
the right side must be zero. Therefore, we have
\begin{align}
\langle0|c(n)=0,~~n\leq 0.
\end{align}
Similarly, we have
\begin{align}
\langle 0|\sum_{i\geq 1}\varphi^*_iz^{-i}=\langle0|\sum_{n\in\mathbb{Z}}(-n)b(n)z^{-n}e^{\beta}\exp\left(\sum_{j>0}\frac{z^{-j}}{-j}\beta_{j}\right).
\end{align}
So 
\begin{align}
\langle0|b(n)=0, n<0.
\end{align}
Therefore, equation \eqref{e:B1} holds.
\begin{remark}
The condition $\langle0|c(0)=c(0)|0\rangle=0$ does not affect the proof of Borchardt's identity since $\varphi_i,~\varphi^*_i,~i\in \mathbb{Z}$ do not depend on $b(0)$.
\end{remark}

\centerline{\bf Acknowledgments}
The work is partially supported by the National Natural Science Foundation of China (grant no. 11531004) and the Simons Foundation (grant no. 523868).


\begin{thebibliography}{10}


\bibitem{Ada2003}Adamovic, D. {\it Classification of irreducible modules of certain subalgebras of free boson vertex algebra}, J. Algebra. 270 
 (2003), 115-132.
\bibitem{Ang2017} Anguelova, I.~I. {\it The two bosonizations of the CKP hierarchy: overview and character identities},
Contemp. Math. 713 (2018), 1-34.
\bibitem{ACJ2014} Anguelova, I.~I., Cox, B., Jurisich, E. {\it N-point locality for vertex operators: Normal ordered products, operator product
expansions, twisted vertex algebras}, J. Pure Appl. Algebra 218 
(2014), 2165-2203.
\bibitem{BaFl2015} Bakalov, B., Fleisher, D. {\it Bosonizations of $\widehat{sl}_{2}$ and integrable hierarchies}, SIGMA 11 (2015), 5-23.
\bibitem{Bor1855}Borchardt, C.~W. {\it Bestimmung der symmetrischen Verbindungen vermittelst ihrer erzeugenden Funktion}, J. Reine Angew. Math. 53 (1855), 193-198.
\bibitem{CWW2015} Chen, M., Wang, S., Wang, X., Wu, K., Zhao, W. {\it On $W_{1+\infty}$ 3-algebra and integrable system}, Nucl. Phys. B, 891 
 (2015), 655-675.
\bibitem{CFM2014} Cox, B., Futorny, V., Martins, R.~A. {\it Free Field Realizations of the Date-Jimbo-Kashiwara-Miwa Algebra}, In: Developments and Retrospectives in Lie Theory. pp. 111-136, Springer, Cham, 2014.
\bibitem{DJKM1981a}Date, E., Jimbo, M., Kashiwara, M., Miwa, T. {\it Operator approach to the kadomtsev-petviashvili equation. Transformation Groups for Soliton Equations III},  J. Phys. Soc. Jan.   50(11) (1981), 3806-3812.
\bibitem{DJKM1981B}Date, E., Jimbo, M., Kashiwara, M., Miwa, T. {\it KP hierarchies of orthogonal and symplectic type. Transformation groups for soliton equations VI}, J. Phys. Soc. Jan.  50 (1981), 3813-3818. 
\bibitem{DKM1982}Date, E., Kashiwara, M., Miwa, T. {\it Vertex operators and tau functions. Transformation groups for soliton equations II}, Proc. Jan. Acad. 57 (1982), 427-507.
\bibitem{DLWY2009}Dong, C., Lam, C.~H., Wang, Q., Yamada, H. {\it The structure of parafermion vertex operator algebras}, J. Algebra. 323 (2009), 783-792. 
\bibitem{DoLe1993} Dong, C., Lepowsky, J. {\it Generalized vertex algebras and relative vertex operators},  Prog. Math., vol. 112,
Birkh\"auser, Boston, 1993.
\bibitem{FF1992} Feigin, B., Frenkel, E. {\it Semi-infinite Weil complex and the Virasoro algebra},  Comm. Math. Phys. 147 (1992), 617-639. 
\bibitem{FKRW1995}Frenkel, E., Kac, V.G., Radul, A.,  Wang, W. {\it $W_{1+\infty}$ and $W(gl_{N})$ with central charge N},  Comm. Math. Phys. 170 (1995), 337-357.
\bibitem{Fr1981} Frenkel, I. {\it Two constructions of affine Lie algebra representations and boson-fermion correspondence in quantum field theory},
    J. Funct. Anal. 44 (1981), 259-327.
\bibitem{FLM} Frenkel, I., Lepowsky, J., and Meurman, A. {\it Vertex operator algebras and the Monster}, Academic Press, 1988.
\bibitem{FMS1985}Friedan, D., Martinec, E., Shenker, S. {\it Conformal invariance, supersymmetry and string theory}, Nucl. Phys. B 271 
(1985), 93-165.
\bibitem{IKO2004} Ishikawa, M., Kawamuko, H., Okada, S. {\it A Pfaffian-Hafnian analogue of Borchardt's identity}, Electron. J. Combin. 12 (2005), Note 9, 8 pp.
\bibitem{Jing1991} Jing, N. {\it Vertex operators, symmetric functions, and the spin group $\Gamma_{n}$}, J. Algebra. 138 (1991), 340-398. 
\bibitem{Jing1995} Jing, N. {\it Boson-fermion correspondence for Hall-Littlewood polynomials}, J. Math. Phys. 36 (1995), 7073-7080. 
\bibitem{KRR2013} Kac, V.~G., Raina, A.~K., Rozhkovskaya, N. {\it Bombay lectures on highest weight representations of infinite dimensional Lie algebras}, Advanced Series in Mathematical Physics, Vol. 29, 2nd. ed., World Scientific, 2013.
\bibitem{KaVa1987} Kac, V.~G., van de Leur, J. {\it Super boson-fermion correspondence}, Ann. Inst. Fourier (Grenoble). 37 (1987), 99-137. 
\bibitem{KaVa1993} Kac, V.~G., van de Leur, J. {\it The n-component KP hierarchy and representation theory}, J. Math. Phys. 44 (1993), 3245-3293. 
\bibitem{LiLep} Li, H., Lepowsky, J. {\it Introduction to vertex operator algebras and their representations}, Springer, New York, 2012.
\bibitem{Li2011} Liszewski, K.~T. {\it The charged free boson integrable hierarchy}, Ph.D. Thesis, North Carolina State University, 2011.
\bibitem{Mac1995} Macdonald, I.~G. {\it Symmetric functions and Hall polynomials}, 2nd edition, Oxford University Press, Oxford U.K., 1995.
\bibitem{VOS2012} van de Leur, J., Orlov A.~Y., Shiota, T. {\it CKP hierarchy, bosonic tau function and bosonization formulae}. SIGMA 8 (2012),no.3, 28pp.
\bibitem{WW} Wang, N., Wu, K. {\it Vertex operators, t-boson model and weighted plane partitions in finite boxes},
Modern Phys. Lett. B 32 (2018), no.5, 15 pp.
\bibitem{Wang1997} Wang, W. {\it ${W}_{1+\infty}$ algebra, ${W}_{3}$ algebra, and Friedan-Martinec-Shenker bosonnization}, Comm. Math. Phys. 195 (1998), 95-111.
\bibitem{Xu2016} Xu, X. {\it Introduction to vertex operator superalgebras and their modules}. Mathematics and its Applications, 456. Kluwer Academic Publishers, Dordrecht, 1998.
\bibitem{ZZ} Zhang, H.,  Zhou, J. {\it Wreath Hurwitz numbers, colored cut-and-join equations, and 2-Toda hierarchy},
Sci. China Math. 55 (2012), 1627-1646.
\end{thebibliography}
\end{document}